\documentclass{amsart}

\usepackage{amsmath}
\usepackage{amsthm}
\usepackage{amsfonts}
\usepackage{dsfont}
\usepackage{color}
\usepackage{enumerate}

\usepackage{datetime}
\shortdate
\yyyymmdddate

\newcommand{\reals}{\mathbb{R}}

\newcommand{\complex}{\mathbb{C}}

\newcommand{\angles}[1]{\left\langle #1 \right\rangle}

\newcommand{\paraa}[1]{\big(#1\big)}
\newcommand{\parab}[1]{\Big(#1\Big)}

\newcommand{\spacearound}[1]{\quad#1\quad}
\newcommand{\equivalent}{\spacearound{\Leftrightarrow}}
\renewcommand{\implies}{\spacearound{\Rightarrow}}

\newtheorem{theorem}{Theorem}[section]

\newtheorem{lemma}[theorem]{Lemma}
\newtheorem{proposition}[theorem]{Proposition}

\theoremstyle{definition}

\theoremstyle{remark}

\numberwithin{equation}{section}

\newcommand{\A}{\mathcal{A}}
\newcommand{\N}{\mathcal{N}}
\renewcommand{\P}{\mathcal{P}}
\newcommand{\gb}{\,\bar{\!g}}
\renewcommand{\d}{\partial}

\newcommand{\eps}{\varepsilon}
\newcommand{\nablab}{\bar{\nabla}}

\newcommand{\vphi}{\varphi}
\renewcommand{\mid}{\mathds{1}}
\newcommand{\D}{\mathcal{D}}
\newcommand{\Der}{\operatorname{Der}}
\renewcommand{\div}{\operatorname{div}}

\newcommand{\TA}{T\mathcal{A}}
\newcommand{\End}{\operatorname{End}}
\newcommand{\Rt}{\tilde{R}}
\newcommand{\alphab}{\bar{\alpha}}
\newcommand{\betab}{\bar{\beta}}
\newcommand{\gammab}{\bar{\gamma}}
\newcommand{\xv}{\vec{x}}
\newcommand{\qb}{\bar{q}}
\newcommand{\rank}{\operatorname{rank}}

\title[]{Curvature and geometric modules of noncommutative
  spheres and tori}

\author{Joakim Arnlind}
\address[Joakim Arnlind]{Dept. of Math.\\
Link\"oping University\\
581 83 Link\"oping\\
Sweden}
\email{joakim.arnlind@liu.se}

\thanks{}

\subjclass[2000]{}

\begin{document}

\begin{abstract}
  When considered as submanifolds of Euclidean space, the Riemannian
  geometry of the round sphere and the Clifford torus may be
  formulated in terms of Poisson algebraic expressions involving the
  embedding coordinates, and a central object is the projection
  operator, projecting tangent vectors in the ambient space onto the
  tangent space of the submanifold.  In this note, we point out that
  there exist noncommutative analogues of these projection operators,
  which implies a very natural definition of noncommutative tangent
  spaces as particular projective modules. These modules carry an
  induced connection from Euclidean space, and we compute its scalar
  curvature.
\end{abstract}

\maketitle

\section{Introduction}

\noindent Linear connections on modules over noncommutative algebras,
and associated differential calculi have been studied from many
different points of view (see
e.g. \cite{c:cstargeometrie,dv:derivationsNonCommutatif,m:linearConnections}
for a derivation based approach). In most cases the definition of the
curvature operator is immediately given as the failure of the
connection to be commutative, in analogy with classical differential
geometry. However, the Ricci and scalar curvature does not come as
easily. In commutative geometry, they arise as contractions over a
basis of the tangent space, which does not always have an apparent
noncommutative analogue (however, see
\cite{cff:gravityncgeometry,mmm:linearConnectionsMatrix,ctzz:riemannianNoncommutativeSurfaces,r:leviCivitaTori}). There
are also more sophisticated definitions relying on the appearance of
the scalar curvature in the expansion of the heat kernel (see
e.g. \cite{cm:modularcurvature}).

In a series of papers
(\cite{ahh:multilinear,ahh:geometryNambu,ahh:psurface}) it was proven
that one may formulate the metric geometry of embedded manifolds in
terms of multi-linear algebraic expressions in the embedding coordinates. For
surfaces, and, in general, almost K\"ahler manifolds, a Poisson
algebraic formulation exists \cite{ah:kahlerpoisson} (see also
\cite{bs:curvatureMatrix}). These results were then used to construct
noncommutative geometric concepts (such as curvature) by simply
replacing Poisson brackets by commutators, and, in the context of
matrix regularizations, these concepts were proven to be useful
\cite{ahh:multilinear}. However, matrix regularizations
rely on a sequence of algebras converging (in a certain sense) to the
commutative algebra of smooth functions on the manifold, and therefore
it was not clear how well adapted these concepts are to a single
noncommutative algebra.

In this note, we will show that the projector of classical geometry,
projecting tangent vectors from the ambient space to the tangent space
of the embedded manifold, has a natural analogue in the noncommutative
algebras of the sphere and the torus. This allows for the definition
of a projective module which one may call the tangent bundle of the
corresponding noncommutative geometry.  Furthermore, an analogue of
the Riemannian connection can be found and the corresponding scalar
curvatures are computed. Note that our approach is in principle not
limited to surfaces, and can be applied to noncommutative algebras
corresponding to submanifolds of any dimension.

\section{Poisson algebraic formulation of surface geometry}

\noindent In \cite{ahh:multilinear} it was shown that the geometry of
embedded Riemannian manifolds can be reformulated in terms of multi-linear
brackets of the embedding coordinates; moreover, in the case of almost
K\"ahler manifolds, a Poisson bracket formulation can be obtained
\cite{ah:kahlerpoisson}. Let us recall the basic facts of this
reformulation, in the case of embedded surfaces.

Let $(\Sigma,g)$ be a $2$-dimensional Riemannian manifold, and let
$\theta$ be a Poisson bivector defining the bracket
\begin{align*}
  \{f,h\} = \theta^{ab}\paraa{\d_af}\paraa{\d_b h}.
\end{align*}
for $f,h\in C^\infty(\Sigma)$. On a 2-dimensional manifold, every
Poisson bivector is of the form $\theta^{ab}=\eps^{ab}/\rho$ for some
density $\rho$ (where $\eps^{12}=-\eps^{21}=1$). The cofactor
expansion of the inverse of a matrix gives the following way of
writing the inverse of the metric
\begin{equation}\label{eq:gthetarel}
  g^{ab} = \frac{1}{g}\eps^{ap}\eps^{bq}g_{ab}\quad\Longrightarrow\quad
  g^{ab} = \frac{\rho^2}{g}\theta^{ap}\theta^{bq}g_{ab},
\end{equation}
which, upon setting $\gamma=\sqrt{g}/\rho$, becomes 
$\gamma^2g^{ab} = \theta^{ap}\theta^{bq}g_{ab}$.

Now, assume that $\Sigma$ is isometrically embedded in a
$m$-dimensional Riemannian manifold $(M,\gb)$, via the embedding
functions $x^1,\ldots,x^m$; i.e.
\begin{align*}
  g_{ab} = \gb_{ij}\paraa{\d_ax^i}\paraa{\d_bx^j}
\end{align*}
where $\d_a = \frac{\d}{\d u^a}$. (Indices $i,j,k,\ldots$ run from $1$
to $m$ and indices $a,b,c,\ldots$ run from $1$ to $2$.) Relation
\eqref{eq:gthetarel} allows one to rewrite geometric object in terms
of Poisson brackets of the embedding functions $x^1,\ldots,x^m$. For
instance, one notes that by defining $\D:T_pM\to T_pM$ as
\begin{align*}
  &{\D^i}_j=\frac{1}{\gamma^2}\{x^i,x^k\}\gb_{kl}\{x^j,x^l\}\gb_{jm}\\
  &\D(X) \equiv {\D^i}_jX^j\d_i
\end{align*}
for $X=X^i\d_i\in T_pM$, one computes
\begin{align*}
  \D(X)^i &= \frac{1}{\gamma^2}\theta^{ab}(\d_ax^i)(\d_bx^k)\gb_{kl}
  \theta^{pq}(\d_px^j)(\d_qx^l)\gb_{jm}X^m\\
  &=\frac{1}{\gamma^2}\theta^{ab}\theta^{pq}g_{bq}(\d_ax^i)(\d_px^j)\gb_{jm}X^m
  =g^{ap}(\d_ax^i)(\d_px^j)\gb_{jm}X^m,
\end{align*}
by using \eqref{eq:gthetarel}. Hence, the map $\D$ is identified as the
orthogonal projection onto $T_p\Sigma$, seen as a subspace of
$T_pM$ and, for convenience, we also introduce the complementary projection as
\begin{equation*}
  \Pi = \mid -\D.
\end{equation*}
Having the projection operator at hand, one may proceed to
develop the theory of submanifolds. For instance, the Levi-Civita
connection $\nabla$ on $\Sigma$ is given by
\begin{align*}
  \nabla_XY = \D\paraa{\nablab_XY}
\end{align*}
where $X,Y\in T_p\Sigma$ and $\nablab$ is the Levi-Civita connection
on $M$. Let us now turn to the particular case we shall be interested
in. Namely, we assume that $(M,\gb_{ij})=(\reals^m,\delta_{ij})$
(which one may always do) and choose $\gamma=1$
(i.e. $\theta^{ab}=\eps^{ab}/\sqrt{g}$). In this setting, the
connection becomes
\begin{align*}
  \nabla_XY^i = {\D^i}_kX(Y^k) 
\end{align*}
where $X(f)$ denotes the action of $X\in T_p\Sigma$ on $f\in
C^\infty(\Sigma)$ as a derivation; as usual, one introduces the
curvature operator as
\begin{align*}
  R(X,Y)Z = \nabla_X\nabla_YZ-\nabla_Y\nabla_XZ-\nabla_{[X,Y]}Z.
\end{align*}
In the non-commutative setting, we shall be interested in a particular
set of derivations; namely, let
\begin{align*}
  \d^i(\cdot) = \{x^i,\cdot\} = \{x^i,x^j\}\d_j(\cdot)
\end{align*}
and set $\nabla^iY^k = \nabla_{\d^i}Y^k = {\D^k}_l\d^i(Y^l)$. With respect to this set of
derivations, one introduces the operator
\begin{align*}
  &\Rt^{ij}(Z) = \nabla^i\nabla^jZ-\nabla^j\nabla^iZ-\nabla_{[\d^i,\d^j]}Z\\
  &\Rt(X,Y)Z = X^iY^i\Rt_{ij}(Z) 
\end{align*}
and computes that 
\begin{align*}
  \Rt^{ij}(Z)^k &= \d^i\paraa{{\D^k}_m}\d^j\paraa{{\D^m}_l}Z^l
  -\d^j\paraa{{\D^k}_m}\d^i\paraa{{\D^m}_l}Z^l\\
  &\equiv \Rt {^{ijk}}_lZ^l.
\end{align*}
The relation to the curvature operator $R$ is given by
\begin{align*}
  R(X,Y)Z = \Rt\paraa{\P(X),\P(Y)}Z
\end{align*}
where $\P(X) = {\P^{i}}_jX^j\d_i$ with $\P^{ij} = \{x^i,x^j\}$.  To
compute the scalar curvature $S$, one has to contract indices of
$R_{ijkl}$ with the projection operator $\D^{ij}$, since one is
summing over a basis of $T_p\Sigma$ (seen as a subspace of $T_pM$);
i.e. $S=\D^{jl}\D^{ik}R_{ijkl}$. Subsequently, the scalar curvature is
given in terms of $\Rt$ as
\begin{align}\label{eq:scCommutative}
  S = 
  \P_{jl}\P_{ik}\Rt^{ijkl},
\end{align}
which is a formula we shall use to define scalar curvature in the
non-commutative setting. Let us now recall how the differential
geometry of the sphere and the torus can be described in terms of Poisson
brackets.

\subsection{The sphere}

One considers the sphere as isometrically embedded in $\reals^3$ via
\begin{align*}
  \xv = (\sin\theta\cos\vphi,\sin\theta\sin\vphi,\cos\theta)
\end{align*}
giving
\begin{align*}
  (g_{ab})=
  \begin{pmatrix}
    1 & 0 \\ 0 & \sin^2\theta
  \end{pmatrix}\quad\text{and}\quad
  \sqrt{g} = \sin\theta.
\end{align*}
By defining
\begin{align*}
  \{f,h\} = \frac{1}{\sqrt{g}}\eps^{ab}\paraa{\d_af}\paraa{\d_bh}
  =\frac{1}{\sin\theta}\eps^{ab}\paraa{\d_af}\paraa{\d_bh}
\end{align*}
one obtains
\begin{align*}
  \{x^i,x^j\} = \eps^{ijk}x^k,
\end{align*}
where $\eps^{ijk}$ is a totally antisymmetric tensor with
$\eps^{123}=1$.  It is then straightforward to show that
\begin{align*}
  &\D^{ij} = \{x^i,x^k\}\{x^j,x^k\} = \delta^{ij}-x^ix^j.\\
  &\Pi^{ij} = x^ix^j\\
  &\Rt^{ijkl} = \paraa{\eps^{ikm}\eps^{jln}-\eps^{jkm}\eps^{iln}}x^mx^n \\
  &S=\P^{jl}\P^{ik}\Rt^{ijkl} = 2.
\end{align*}

\subsection{The torus}

\noindent The Clifford torus is considered as embedded in $\reals^4$ via
\begin{align*}
  \xv = \frac{1}{\sqrt{2}}(\cos\vphi_1,\sin\vphi_1,\cos\vphi_2,\sin\vphi_2) 
\end{align*}
giving rise to the induced metric
\begin{align*}
  (g_{ab}) = \frac{1}{2}
  \begin{pmatrix}
    1 & 0 \\ 0 & 1 
  \end{pmatrix}
  \quad\text{with}\quad
  \sqrt{g} = \frac{1}{2}.
\end{align*}
By defining
\begin{align*}
  \{f,h\} = \frac{1}{\sqrt{g}}\eps^{ab}\paraa{\d_af}\paraa{\d_bh}
  =2\eps^{ab}\paraa{\d_af}\paraa{\d_bh}
\end{align*}
one obtains
\begin{align*}
  \paraa{\{x^i,x^j\}} = 2
  \begin{pmatrix}
    0 & 0 & x^2x^4 & -x^2x^3 \\
    0 & 0 & -x^1x^4 & x^1x^3 \\
    -x^2x^4 & x^1x^4 & 0 & 0 \\
    x^2x^3 & -x^1x^3 & 0 & 0
  \end{pmatrix},
\end{align*}
from which it follows that
\begin{align*}
  &\D = 2
  \begin{pmatrix}
    (x^2)^2 & -x^1x^2 & 0 & 0 \\
    -x^1x^2 & (x^1)^2 & 0 & 0 \\
    0 & 0 & (x^4)^2 & -x^3x^4 \\
    0 & 0 & -x^3x^4 & (x^3)^2
  \end{pmatrix}\\
  &\Pi = 2
  \begin{pmatrix}
    (x^1)^2 & x^1x^2 & 0 & 0 \\
    x^1x^2 & (x^2)^2 & 0 & 0 \\
    0 & 0 & (x^3)^2 & x^3x^4 \\
    0 & 0 & x^3x^4 & (x^4)^2
  \end{pmatrix}.
\end{align*}
Furthermore, a straightforward computation yields $\Rt^{ijkl}=0$.

\section{Connections and curvature}

\noindent Let $\A$ be an associative $\ast$-algebra. A
$\ast$-derivation is a derivation $\d$ such that
$\d(a)^\ast=\d(a^\ast)$ for all $a\in\A$; by $\Der(\A)$ we shall
denote the vector space (over $\reals$) of $\ast$-derivations of $\A$.  Moreover,
assume that there exists a projector $\D$, acting on the (right) free
module $\A^m$, i.e. $\D\in\End(\A^m)$ and $\D^2=\D$, and by $\TA$ we
denote the corresponding (finitely generated) projective module
$\D(\A^m)$. Letting $\{e_i\}_{i=1}^m$ denote the canonical basis of
$\A^m$, one can write the action of $\D$ as
\begin{equation*}
  \D(U) = e_i{\D^i}_jU^j,
\end{equation*}
for $U=e_iU^i$ (note that there is no difference between lower and
upper indices, but let us keep the notation that is familiar from
differential geometry for now). We also introduce the complementary
projection $\Pi=\mid-\D$. Moreover, for every $\d\in\Der(\A)$ one
defines
\begin{align*}
  \nablab_\d U = e_k\d(U^k)
\end{align*}
corresponding (in the commutative case) to the connection in the
``ambient'' space. Note that the two arguments of the connection are
not on equal footing; one is a derivation and the other one belongs to
a free module. The map $\nablab_\d$ is an affine connection on $\A^m$
in the sense that 
\begin{align}\label{eq:affineConnection}
  \begin{split}
    &\nablab_\d(U+V) = \nablab_\d U+\nablab_\d V\\
    &\nablab_{c\d}U=c\nablab_{\d} U\\
    &\nablab_{\d+\d'}U = \nablab_\d U+\nablab_{\d'}U\\
    &\nablab_{\d}(Ua) = \paraa{\nablab_\d U}a+U\d(a)    
  \end{split}
\end{align}
for $a\in\A$, $c\in\reals$, $\d,\d'\in\Der(\A)$ and
$U,V\in\A^m$. Furthermore, by introducing a metric on $\A^m$ via
\begin{align}\label{eq:metricDef}
  \angles{U,V} = (U_i)^\ast V^i,
\end{align}
for $U=e_iU^i\in\A^m$ and $V=e_iV^i\in\A^m$, it is straightforward to
show that $\nablab$ is a metric connection; i.e.
\begin{align*}
  \d\angles{U,V}-\angles{\nablab_\d U,V}-\angles{U,\nablab_\d V}=0
\end{align*}
for all $\d\in\Der(\A)$ and $U,V\in\A^m$. As for ordinary manifolds,
one proceeds to define a connection on $\TA=\D(\A^m)$ by setting
\begin{align*}
  \nabla_\d U = \D\paraa{\nablab_\d U} = e_i{\D^i}_j\d(U^j)
\end{align*}
for $\d\in\Der(\A)$ and $U=e_iU^i\in\TA$; it follows that $\nabla$
satisfies the requirements \eqref{eq:affineConnection} of an affine
connection. We shall assume that $\D$ is symmetric with respect to the
metric introduced in \eqref{eq:metricDef};
i.e. $\angles{\D(U),V}=\angles{U,\D(V)}$ for all $U,V\in\A^m$. In this
case, $\nabla$ will be a also be a metric connection\footnote{While
  preparing this paper we became aware of
  \cite{zz:projectiveModuleEmbedded} which treats connections on
  projective modules in a somewhat similar way.}. 

Now, let us choose a set of elements $X^1,\ldots,X^m\in\A$ together
with their associated inner $\ast$-derivations
\begin{align*}
  \d^i(a) = \frac{1}{i\hbar}[X^i,a]
\end{align*}
for an arbitrary parameter $\hbar\in\reals$ (in the current setting,
one might as well put $\hbar=1$, but it will be convenient later on).
In analogy with classical geometry, one should think of the $X^i$'s as
embedding coordinates of a manifold into $\reals^m$. A different
choice of embedding does in general lead to a different induced metric
on the submanifold. Therefore, the choice of $X^i$'s amount to a
choice of the metric structure on the algebra.

With the help of the above derivations we introduce, for $U\in\TA$,
\begin{align*}
  \Rt^{ij}(U) = \nabla^i\nabla^jU-\nabla^j\nabla^iU-\nabla_{[\d^i,\d^j]}U,
\end{align*}
where $\nabla^iU = \nabla_{\d^i}U$.  That $\Rt^{ij}$ is a
module homomorphism becomes clear from the following result:
\begin{proposition}
  For $U=e_iU^i\in\TA$ it holds that
  \begin{align*}
    \Rt^{ij}(U) = e_k\parab{\d^i\paraa{{\D^k}_m}\d^j\paraa{{\D^m}_l}
    -\d^j\paraa{{\D^k}_m}\d^i\paraa{{\D^m}_l}}U^l.
  \end{align*}
\end{proposition}

\begin{proof}
  Let $U\in\TA$ with $U=e_iU^i$. Using that $\D(U)=U$ and Leibnitz rule
  one obtains
  \begin{align*}
    \nabla^i\nabla^j(U) &= e_k{\D^k}_l\d^i\paraa{{\D^l}_m\d^j(U^m)}
    =e_k{\D^k}_l\d^i({\D^l}_m)\d^j(U^m)+e_k{\D^k}_m\d^i\d^j(U^m),
  \end{align*}
  and one may rewrite the first term as
  \begin{align*}
    e_k{\D^k}_l\d^i({\D^l}_m)\d^j(U^m) &=
    e_k\d^i({\D^k}_m)\d^j(U^m)-e_k\d^i({\D^k}_l){\D^l}_m\d^j(U^m)\\
    &=e_k\d^i({\D^k}_l)\d^j({\D^l}_m)U^m.
  \end{align*}
  Hence, it holds that
  \begin{align*}
    \nabla^i\nabla^j(U) = 
    e_k\d^i({\D^k}_l)\d^j({\D^l}_m)U^m+e_k{\D^k}_m\d^i\d^j(U^m),
  \end{align*}
  from which the desired formula follows for $\Rt^{ij}$.
\end{proof}

\noindent Consequently, one introduces
\begin{align*}
  \Rt {^{ijk}}_l = \d^i\paraa{{\D^k}_m}\d^j\paraa{{\D^m}_l}
    -\d^j\paraa{{\D^k}_l}\d^i\paraa{{\D^l}_m}
\end{align*}
giving $\Rt^{ij}(U)=e_k\Rt {^{ijk}}_lU^l$. In analogy with formula 
(\ref{eq:scCommutative}) we define the scalar curvature of $\nabla$ as
\begin{align}
  S = \P_{jl}\P_{ik}\Rt^{ijkl}
\end{align}
where $\P^{ij} = \frac{1}{i\hbar}[X^i,X^j]$.

Furthermore, let us introduce the divergence of an element $U\in\TA$ as:
\begin{align*}
  \div(U) = \nabla_iU^i = \D_{ik}\d^i(U^k)\in\A. 
\end{align*}
Let $\phi:\A\to\complex$ be a $\complex$-linear functional such that
$\phi(ab)=\phi(ba)$ for all $a,b\in\A$; we shall refer to such a
linear functional as a \emph{trace}. Moreover, a trace $\phi$ is said to
be \emph{closed} if it holds that
\begin{align*}
  \phi\paraa{\div(U)} = 0
\end{align*}
for all $U\in\TA$.

Let us, for later convenience, slightly rewrite the condition that $\phi$ is
a trace.

\begin{lemma}\label{lemma:phiClosedEquiv}
  A trace $\phi$ is closed if and only if it holds that
  \begin{align}
    \phi\paraa{[X^i,\Pi_{ik}]U^k} = 0
  \end{align}
  for all $U = e_iU^i\in\TA$.
\end{lemma}

\begin{proof}
  Using that $\phi$ is a trace, one computes that
  \begin{align*}
    \phi\paraa{\div(U)} &= 
    \phi\paraa{\D_{ik}\d^i(U^k)} =
    \phi\paraa{\d^i(\D_{ik}U^k)-\d^i(\D_{ik})U^k}\\
    &=\phi\paraa{\d^i(\D_{ik})U^k}
    =\phi\paraa{\d^i(\Pi_{ik})U^k}
    =\frac{1}{i\hbar}\phi\paraa{[X^i,\Pi_{ik}]U^k},
  \end{align*}
  from which the statement follows.
\end{proof}

\section{The fuzzy sphere}

\noindent For our purposes, we shall define the fuzzy sphere
\cite{h:phdthesis,m:fuzzysphere} as a (unital associative)
$\ast$-algebra $S^2_\hbar$ on three hermitian generators $X^1,X^2,X^3$
satisfying the following relations:
\begin{align*}
  &[X^i,X^j] = i\hbar\eps^{ijk}X^k\\
  &\paraa{X^1}^2+\paraa{X^2}^2+\paraa{X^3}^2=\mid.
\end{align*}
It is easy to see that, by setting $\Pi^{ij}=X^iX^j$ as a
non-commutative analogue of the classical projection operator, it
holds that
\begin{align*}
  (\Pi^2)^{ij} = \Pi^{ik}\Pi^{kj}  = X^iX^kX^kX^j = X^i\mid X^j = \Pi^{ij},
\end{align*}
which shows that $\Pi$ is a projection operator when considered as an
endomorphism of the free module $(S^2_\hbar)^3$; moreover, $\Pi$ is symmetric since
$(\Pi^{ij})^\ast=X^jX^i=\Pi^{ji}$. Let us note that the similarity with the
commutative formulas is even stronger; namely, one easily checks that
\begin{align*}
  \D^{ij} = \delta^{ij}\mid-X^iX^j = \frac{1}{(i\hbar)^2}[X^j,X^k][X^i,X^k].
\end{align*}
One may proceed and define a connection $\nabla$ on
$TS^2_\hbar=\D\paraa{(S^2_\hbar)^3}$ as in the previous section, and since
the projection operator is symmetric, this is a metric connection. As
it will be helpful in computations, let us remind ourselves of a few
identities involving $\eps^{ijk}$:
\begin{align*}
  &\eps^{ijk}\eps^{imn} = \delta^{jm}\delta^{kn}-\delta^{jn}\delta^{km}\qquad\quad
  \eps^{ikl}\eps^{jkl} = 2\delta^{ij}\\
  &\eps^{ijk}X^jX^k = i\hbar X^i\qquad\qquad\quad
  \eps^{ijk}X^iX^jX^k = i\hbar\mid.
\end{align*}

\noindent Let us now compute the curvature of $\nabla$.

\begin{proposition}
  For the fuzzy sphere, it holds that
  \begin{align*}
    \Rt^{ijkl} &= 
    \paraa{\eps^{ikp}\eps^{jlq}-\eps^{jkp}\eps^{ilq}}X^pX^q
    -i\hbar\eps^{jlq}X^kX^iX^q
    -i\hbar\eps^{jkp}X^pX^iX^l\\
    &\quad+i\hbar\eps^{ikp}X^pX^jX^l
    +i\hbar\eps^{ilq}X^kX^jX^q
    +i\hbar\eps^{ijp}X^kX^pX^l\\
    S &=\paraa{2-3\hbar^2+h^4}\mid.
  \end{align*}
\end{proposition}

\begin{proof}
  The proof consists of a straightforward computation. Starting from
  \begin{align*}
    \Rt^{ijkl} &= \d^i\paraa{\D^{km}}\d^j\paraa{\D^{ml}}
    -\d^j\paraa{\D^{km}}\d^i\paraa{\D^{ml}}\\
    &=\d^i\paraa{\Pi^{km}}\d^j\paraa{\Pi^{ml}}
    -\d^j\paraa{\Pi^{km}}\d^i\paraa{\Pi^{ml}}\\
    &=\frac{1}{(i\hbar)^2}[X^i,X^kX^m][X^j,X^mX^l]
    -\frac{1}{(i\hbar)^2}[X^j,X^kX^m][X^i,X^mX^l]
  \end{align*}
  one expands the expression, using $[X^i,X^j]=i\hbar\eps^{ijk}X^k$
  and the $\eps$-identities we previously recalled, to obtain
  \begin{align*}
    \Rt^{ijkl} = 
    &\paraa{\eps^{ikp}\eps^{jlq}-\eps^{jkp}\eps^{ilq}}X^pX^q
    -i\hbar\eps^{jlq}X^kX^iX^q
    -i\hbar\eps^{jkp}X^pX^iX^l\\
    &+i\hbar\eps^{ikp}X^pX^jX^l
    +i\hbar\eps^{ilq}X^kX^jX^q
    +i\hbar\eps^{ijp}X^kX^pX^l.
  \end{align*}
  From this expression one derives
  \begin{align*}
    \P^{ik}\Rt^{ijkl} &= 
    (1-\hbar^2-\hbar^4)\eps^{jlm}X^m
    +i\hbar(1-3\hbar^2)X^jX^l+i\hbar^3\delta^{jl}
  \end{align*}
  again by using the appropriate identities. Finally, the scalar
  curvature is computed
  \begin{align*}
    S &= \P^{jl}\P^{ik}\Rt^{ijkl}\\ 
    &= 
    (1-\hbar^2-\hbar^4)\eps^{jlk}X^k\eps^{jlm}X^m
    +i\hbar(1-3\hbar^2)\eps^{jlk}X^kX^jX^l+i\hbar^3\eps^{jlk}X^k\delta^{jl}\\
    &= 2(1-\hbar^2-\hbar^4)\mid +i\hbar(1-3\hbar^2)i\hbar\mid
    =\paraa{2-3h^2+h^4}\mid,
  \end{align*}
  which proves the statement.
\end{proof}

\noindent Let us now show that every trace on the fuzzy sphere is closed.

\begin{proposition}
  Let $\phi$ be a trace on $S^2_\hbar$. Then $\phi$ is closed.
\end{proposition}

\begin{proof}
  Starting from the formula in Lemma \ref{lemma:phiClosedEquiv} one
  computes
  \begin{align*}
    \frac{1}{i\hbar}\phi\paraa{[X^i,\Pi^{ik}]U^k}
    &= \frac{1}{i\hbar}\phi\paraa{[X^i,X^iX^k]U^k}
    = \frac{1}{i\hbar}\phi\paraa{X^i[X^i,X^k]U^k}\\
    &= \phi\paraa{\eps^{ikl}X^iX^lU^k}
    =-i\hbar\phi\paraa{X^kU^k}
  \end{align*}
  Now, since $U\in\TA$ it holds that $\Pi(U)=0$, which is equivalent to
  \begin{align*}
    X^iX^kU^k = 0
  \end{align*}
  for $i=1,2,3$. Multiplying the above equation by $X^i$ from the
  left, and summing over $i$ yields
  \begin{align*}
    0 = X^iX^iX^kU^k = X^kU^k.
  \end{align*}
  Thus, $X^kU^k=0$, which proves that $\phi$ is closed.
\end{proof}

\noindent Note that one may easily compute the rank of the
module $TS^2_\hbar$ and its complementary module $\N =
\Pi\paraa{(S^2_\hbar)^3}$ as the trace of the corresponding
projections; i.e.
\begin{align*}
  &\rank(TS^2_\hbar) = \sum_{i=1}^3\D^{ii} =
  \sum_{i=1}^3\paraa{\delta^{ii}\mid-X^iX^i)}=2\mid\\
  &\rank(\N) = \sum_{i=1}^3\Pi^{ii}=\sum_{i=1}^3X^iX^i = \mid,
\end{align*}
corresponding to the geometric dimensions in the commutative
setting. Moreover, the module $\N$ turns out to be a free module.

\begin{proposition}
  The module $\N = \Pi\paraa{(S^2_\hbar)^3}$ is a free module of rank
  $1$, and it is generated by $X = e_iX^i$.
\end{proposition}

\begin{proof}
  An element $N=e_iN^i\in\N$ satisfies
  \begin{align*}
    X^iX^jN^j = N^i
  \end{align*}
  for $i=1,2,3$, which implies that there exists an element
  $a=X^jN^j\in\A$ such that $N = e_iX^i\cdot a$. This proves that
  $e_iX^i$ generates $\N$. Furthermore, one computes that
  \begin{align*}
    0 = X^ia \implies 0 = X^iX^ia\implies 0 = a,
  \end{align*}
  which shows that $\N$ is indeed a free module.
\end{proof}

\section{The non-commutative torus}

\noindent The non-commutative torus $\A_\theta$ (for
$\theta\in\reals$) \cite{c:cstargeometrie} is defined as the unital
associative $\ast$-algebra on two unitary generators $U,V$ satisfying
the following relation
\begin{align*}
  VU = qUV
\end{align*}
with $q=e^{2i\theta}$. Defining hermitian elements
\begin{align*}
  &X^1 = \frac{1}{2\sqrt{2}}(U^\ast+U)\qquad 
  X^2 = \frac{i}{2\sqrt{2}}(U^\ast-U)\\
  &X^3 = \frac{1}{2\sqrt{2}}(V^\ast+V)\qquad
  X^4 = \frac{i}{2\sqrt{2}}(V^\ast-V)
\end{align*}
it follows that
\begin{align}
  &[X^1,X^2] = [X^3,X^4] = 0\label{eq:X1X2X3X4Commute}\\
  &[X^1,X^3] = i\hbar\paraa{X^2X^4+X^4X^2}\label{eq:X1X3Commutator}\\
  &[X^2,X^4] = i\hbar\paraa{X^1X^3+X^3X^1}\label{eq:X2X4Commutator}\\
  &[X^1,X^4] = -i\hbar\paraa{X^2X^3+X^3X^2}\label{eq:X1X4Commutator}\\
  &[X^2,X^3] = -i\hbar\paraa{X^1X^4+X^4X^1}\label{eq:X2X3Commutator}\\
  &(X^1)^2 + (X^2)^2 = (X^3)^2+(X^4)^2 = \frac{1}{2}\mid.\label{eq:sumSqHalf}
\end{align}
with $\hbar=\tan\theta$. Conversely, one can show that the above
relations imply that 
\begin{align*}
  U = \sqrt{2}\paraa{X^1+iX^2}\qquad
  V = \sqrt{2}\paraa{X^3+iX^4}
\end{align*}
are unitary elements satisfying $VU=qUV$. Namely, since
$[X^1,X^2]=[X^3,X^4]=0$ is follows immediately that
$[U,U^\ast]=[V,V^\ast]=0$, and from $(X^1)^2+(X^2)^2=\mid/2$ and
$(X^3)^2+(X^4)^2=\mid/2$ it follows that 
\begin{align*}
  UU^\ast+U^\ast U=VV^\ast+V^\ast V = 2\mid,  
\end{align*}
which, together with $[U,U^\ast]=[V,V^\ast]=0$, implies that $U$ and
$V$ are unitary. Furthermore, noting that
(\ref{eq:X1X3Commutator})--(\ref{eq:X2X3Commutator})
implies that 
\begin{align*}
  &X^3X^1=\cos(2\theta)X^1X^3-i\sin(2\theta)X^2X^4\\
  &X^4X^2=\cos(2\theta)X^2X^4-i\sin(2\theta)X^1X^3\\
  &X^4X^1=\cos(2\theta)X^1X^4+i\sin(2\theta)X^2X^3\\
  &X^3X^2=\cos(2\theta)X^2X^3+i\sin(2\theta)X^1X^4
\end{align*}
one readily shows that $VU=qUV$.

Since there is a natural split of the $X^i$'s into two groups, let us
develop some notation reflecting this fact.  Greek indices
$\alpha,\beta,\ldots$ will take values in $\{1,2\}$ and ``barred''
indices $\alphab,\betab,\ldots$ take values in $\{3,4\}$. With this
notation, the projector $\Pi$ may be defined as (in analogy with the
classical formula)
\begin{align*}
  &\Pi^{\alpha\alphab}=\Pi^{\alphab\alpha} = 0\\
  &\Pi^{\alpha\beta} = 2X^\alpha X^\beta\qquad \Pi^{\alphab\betab} = 2X^{\alphab} X^{\betab}
\end{align*}
and one checks that $\Pi^2 = \Pi$, as well as
$(\Pi^{ij})^\ast=\Pi^{ji}$. For the forthcoming computations, one
notes that
\begin{align}
  &[X^\alpha,X^\beta] = [X^{\alphab},X^{\betab}] = 0\label{eq:XXalphabetazero}\\
  &[X^\alpha,\Pi^{\beta\gamma}]
  =[X^{\alphab},\Pi^{\betab\gammab}] = 0.\label{eq:XPiZero}
\end{align}

\begin{proposition}
  The curvature $\Rt$ of $\A_\theta$ vanishes; i.e
  \begin{align*}
    \Rt^{ijkl} = 0
  \end{align*}
  for $i,j,k,l\in\{1,2,3,4\}$.
\end{proposition}

\begin{proof}
  Using \eqref{eq:XXalphabetazero} and \eqref{eq:XPiZero}, it is easy
  to see that 
  \begin{align*}
    \Rt^{\alpha\alphab kl}=\Rt^{\alpha\beta\gamma k}
    =\Rt^{\alpha\beta k \gamma}=\Rt^{\alphab\betab\gammab k}
    =\Rt^{\alphab\betab k\gammab}
    =0.
  \end{align*}
  Thus, it remains to show that
  $\Rt^{\alpha\beta\alphab\betab}=\Rt^{\alphab\betab\alpha\beta}=0$; let
  us outline the calculation for $\Rt^{\alpha\beta\alphab\betab}$. It
  turns out to be slightly easier to perform the computation using
  variables $U$ and $V$ instead of $X^i$, and one writes
  \begin{align*}
    &X^\alpha =
    \frac{i^{\alpha-1}}{2\sqrt{2}}\paraa{U^\ast+(-1)^{\alpha-1}U}
    \qquad(\alpha=1,2)\\
    &X^{\alphab} = \frac{-i^{\alphab-1}}{2\sqrt{2}}
    \paraa{V^\ast+(-1)^{\alphab-1}V}
    \qquad(\alphab=3,4).
  \end{align*}
  Since 
  \begin{equation}\label{eq:torusRab}
    \begin{split}
      \Rt^{\alpha\beta\alphab\betab} &=
      \d^\alpha\paraa{\Pi^{\alphab\gammab}}\d^\beta\paraa{\Pi^{\gammab\betab}}
      -\d^\beta\paraa{\Pi^{\alphab\gammab}}\d^\alpha\paraa{\Pi^{\gammab\betab}}\\
      &= \frac{4}{(i\hbar)^2}[X^\alpha,X^{\alphab}X^{\gammab}]
      [X^\beta,X^{\gammab}X^{\betab}]
      -\frac{4}{(i\hbar)^2}[X^\beta,X^{\alphab}X^{\gammab}]
      [X^\alpha,X^{\gammab}X^{\betab}],
    \end{split}
  \end{equation}
  let us start by computing $2[X^\alpha,X^{\alphab}X^{\gammab}]$:
  \begin{align*}
    2[X^\alpha,X^{\alphab}X^{\gammab}] &=
    \frac{i^{\alpha+\alphab+\gammab+1}}{8\sqrt{2}}
    [U^\ast+(-1)^{\alpha-1}U,\paraa{V^\ast+(-1)^{\alphab-1}V}
    \paraa{V^\ast+(-1)^{\gammab-1}V}]\\
    &= \frac{i^{\alpha+\alphab+\gammab+1}}{8\sqrt{2}}\Big(
      (1-q^2)U^\ast(V^\ast)^2
      +(-1)^{\alpha+\alphab+\gammab-1}(1-q^2)UV^2\\
      &\qquad+(-1)^{\alpha-1}(1-\qb^2)U(V^\ast)^2
      +(-1)^{\alphab+\gammab}(1-\qb^2)U^\ast V^2\Big)
  \end{align*}
  by using $VU=qUV$ and $V^\ast U=\qb UV^\ast$. Subsequently, using
  this result, one computes (sum over $\gammab$ implied)
  \begin{align*}
    [X^\alpha,X^{\alphab}X^{\gammab}][X^\beta,&X^{\gammab}X^{\betab}]=
    \frac{-i^{\alpha+\beta+\alphab+\betab}}{64}\Big(
    \qb^2(1-q^2)^2\paraa{(-1)^{\beta+\betab-1}+(-1)^{\alpha+\alphab-1}}\mid\\
    &+q^2(1-\qb^2)\paraa{(-1)^{\alpha+\betab-1}+(-1)^{\alphab+\beta-1}}\mid\\
    &+(1-q^2)(1-\qb^2)\paraa{(-1)^{\alphab}\qb^2+(-1)^{\betab}q^2}(U^\ast)^2\\
    &+(1-q^2)(1-\qb^2)\paraa{(-1)^{\alpha+\beta+\alphab}q^2+(-1)^{\alpha+\beta+\betab}\qb^2}U^2
    \Big)
  \end{align*}
  where many terms vanish due to the fact that anything proportional
  to $(-1)^{\gammab}$ cancel when summing over $\gammab$. Since
  \begin{align*}
    q^2(1-\qb^2)^2 = q^2+\qb^2-2 = \qb^2(1-q^2)^2
  \end{align*}
  one notes that the previous expression is \emph{symmetric} with respect
  to interchanging $\alpha$ and $\beta$, which implies, via
  \eqref{eq:torusRab}, that $\Rt^{\alpha\beta\alphab\betab}=0$.
\end{proof}

\noindent Let us now show that, as for the fuzzy sphere, every trace
on $\A_\theta$ is closed.

\begin{proposition}
  Let $\phi$ be a trace on $\A_\theta$. Then $\phi$ is closed.
\end{proposition}

\begin{proof}
  Let us prove that $[X^i,\Pi^{ik}]=0$ for $k=1,2,3,4$. Lemma
  \ref{lemma:phiClosedEquiv} then implies that $\phi$ is closed.
  First, assume that $k=\beta$:
  \begin{align*}
    [X^i,\Pi^{i\beta}] = [X^\alpha,\Pi^{\alpha\beta}] + 
    [X^{\alphab},\Pi^{\alphab\beta}] = 0,
  \end{align*}
  since $\Pi^{\alphab\beta}=0$ and
  $[X^{\alpha},\Pi^{\alpha\beta}]=0$. Similarly, when $k=\betab$ one
  obtains
  \begin{align*}
    [X^i,\Pi^{i\betab}] = [X^\alpha,\Pi^{\alpha\betab}]
    +[X^{\alphab},\Pi^{\alphab\betab}] = 0,
  \end{align*}
  which proves that $[X^i,\Pi^{ik}]=0$.
\end{proof}

\noindent The rank of $T\A_\theta=\D(\A_\theta^4)$ and
$\N\A_\theta=\Pi(\A_\theta^4)$ can again be computed via the trace of
the corresponding projection operators:
\begin{align*}
  &\rank(T\A_\theta) = \sum_{i=1}^4\paraa{\delta^{ii}\mid-2X^iX^i} =
  2\mid\\ 
  &\rank(\N\A_\theta) = \sum_{i=1}^4\paraa{2X^iX^i} = 2\mid. 
\end{align*}

\noindent Now, let us show that, in fact, both $T\A_\theta$ and
$\N\A_\theta$ are free modules.

\begin{proposition}
  The module $T\A_\theta = \D(\A_\theta^4)$ is a free module of rank
  $2$, with a basis given by $E_1 = -e_1X^2+e_2X^1$ and $E_2 =
  -e_3X^4+e_4X^3$.
\end{proposition}

\begin{proof}
  First of all, it is easy to check that $\D(E_1)=E_1$ and
  $\D(E_2)=E_2$, which implies that $E_1,E_2\in T\A_\theta$. Moreover,
  $E_1$ and $E_2$ are linearly independent, since
  \begin{align*}
    &E_1a + E_1b = 0\implies
    (-X^2a,X^1a,-X^4b,X^3b)=(0,0,0,0)\implies\\
    &\begin{cases}
      \paraa{(X^1)^2+(X^2)^2}a = 0\\
      \paraa{(X^3)^2+(X^4)^2}b = 0
    \end{cases}\implies
    a=b=0.
  \end{align*}
  Let us now show that $E_1$ and $E_2$ span $T\A_\theta$. By
  definition of $T\A_\theta$ there exists, for every $Y\in
  T\A_\theta$, and element $U\in\A_\theta^4$ such that $Y=\D(U)$. One
  readily computes that
  \begin{align*}
    &\D(U)^1 = -X^2\paraa{2X^1U^2-2X^2U^1}\\
    &\D(U)^2 = X^1\paraa{2X^1U^2-2X^2U^1}\\
    &\D(U)^3 = -X^4\paraa{2X^3U^4-2X^4U^3}\\
    &\D(U)^4 = X^3\paraa{2X^3U^4-2X^4U^3};
  \end{align*}
  that is, for every $U\in\A_\theta^4$, there exist $a,b\in\A_\theta$
  such that $\D(U)=E_1a+E_2b$, which implies that $E_1$ and $E_2$ span
  $T\A_\theta$.
\end{proof}

\begin{proposition}
  The module $\N\A_\theta = \Pi(\A_\theta^4)$ is a free module of
  rank $2$, with a basis given by $N_{\pm} = e_1X^1+e_2X^2\pm e_3X^3\pm
  e_4X^4$.
\end{proposition}

\begin{proof}
  It is easy to check that $\Pi(N_+)=N_+$ and $\Pi(N_-)=N_-$ which
  shows that they are indeed elements of $\N\A_\theta$. Thus, every
  element of the form
  \begin{align}
    N = e_1X^1a + e_2X^2a + e_3X^3b + e_4X^4b\label{eq:NXpXm}
  \end{align}
  is an element of $\N\A_\theta$.  Now, let $N=e_iN^i\in\A_\theta^4$
  such that $\Pi(N)=N$, which is equivalent to
  \begin{align*}
    &\Pi^{\alpha\beta}N^\beta = N^\alpha\equivalent
    2X^\alpha X^\beta N^\beta = N^\alpha\\
    &\Pi^{\alphab\betab}N^{\betab} = N^{\alphab}\equivalent
    2X^{\alphab}X^{\betab}N^{\betab} = N^{\alphab}.
  \end{align*}
  This immediately implies that $N$ can be written in the form
  \eqref{eq:NXpXm}. Thus, the elements $N_+$ and $N_-$ generate
  $\N\A_\theta$. Next, assume that
  \begin{align*}
    N = e_1X^1a + e_2X^2a + e_3X^3b + e_4X^4b=0,
  \end{align*}
  which is equivalent to $X^\alpha a = 0$ and
  $X^{\alphab}b=0$. Multiplying by $X^\alpha$ and $X^{\alphab}$,
  respectively, and summing over the index yields $a=b=0$. Hence,
  $N_+$ and $N_-$ is a basis for the module $\N\A_\theta$. 
\end{proof}

\noindent Finally, we note that the set $\{E_1,E_2,N_+,N_-\}$ is a set of mutually orthogonal
elements with respect to the metric $\angles{\cdot,\cdot}$.

\section*{Acknowledgment}

\noindent I would like to thank J. Choe and J. Hoppe for discussions,
and the Korea Institute for Advanced Study and Sogang University for
hospitality and financial support.


\bibliographystyle{alpha}
\bibliography{fuzzytsp}  

\end{document}